\documentclass[11pt]{amsart}
\usepackage{amssymb}
\usepackage{verbatim}
\usepackage{hyperref}
\usepackage{color}
\usepackage{enumitem}

\usepackage{enumitem}

\newtheorem{prop}{Proposition}
\newtheorem{lemma}[prop]{Lemma}

\newtheorem{theorem}[prop]{Theorem}

\newtheorem{definition}[prop]{Definition}
\newtheorem{notation}[prop]{Notation}

\newcommand{\R}{\mathbb R}

\newcommand{\Z}{\mathbb Z}

\newcommand{\T}{\mathbb T}

\newcommand{\w}{\omega}
\newcommand{\e}{\varepsilon}
\newcommand{\g}{\gamma}

\newcommand{\m}{\textfrak{m}}

\newcommand{\x}{{\bf{x}}}
\renewcommand{\v}{{\bf{v}}}

\newcommand{\mc}{\mathcal}

\def\set4{\mathcal I}
\def\tup14{(1,2,3,4)}

\begin{document}

\title{Regularized Brascamp-Lieb inequalities and an application}
\author{Dominique Maldague}
\address{Department of Mathematics\\
Massachusetts Institute of Technology\\
Cambridge, MA 02142-4307, USA}
\email{dmal@mit.edu}

\date{\today}

\maketitle

\begin{abstract}
We present a certain regularized version of Brascamp-Lieb inequalities studied by Bennett, Carbery, Christ, and Tao. Using the induction-on-scales method of Guth, these regularized inequalities lead to a generalization of the multilinear Kakeya inequality.
\end{abstract}

By Brascamp-Lieb inequalities, we mean inequalities of the form 
\begin{equation}  \int_{\R^n}\prod_{j=1}^Jf_j(\pi_j(\x))^{p_j} d\x \le C \prod_{j=1}^J\left(\int_{\R^{n_j}}f_j(\x_j)d\x_j\right)^{p_j}. \label{HBL} \end{equation}
where the $f_j$ are nonnegative, measurable functions, $\vec{\pi}=(\pi_j)$ is a $J$-tuple of linear maps from $\R^n$ to $\R^{n_j}$, and $\vec{p}\in[0,1]^J$. A related quantity is the smallest constant $C$ which satisfies the above inequality for all nonnegative input functions $f_j\in L^1(\R^{n_j})$, denoted by $\text{BL}(\vec{\pi},\vec{p})$. Locally bounded properties of $\text{BL}(\vec{\pi},\vec{p})$ lead to multilinear Kakeya inequalities: Guth \cite{guth} first did this for a special case $(\vec{\pi}_0,\vec{p}_0)$ and Bennett, Bez, Flock, and Lee \cite{stabBL} extended Guth's approach to data $(\vec{\pi},\vec{p})$ satisfying $\text{BL}(\vec{\pi},\vec{p})<\infty$. The motivation for this paper is to formulate a regularized version of $\text{BL}(\vec{\pi},\vec{p})$ which leads to multilinear Kakeya inequalities in the case that we do not have finite Brascamp-Lieb data (see below for further discussion of multilinear Kakeya).

Some of the most fundamental inequalities of this form are H\"{o}lder's inequality and Young's convolution inequality. In \cite{BL}, Brascamp and Lieb determined the optimal version of Young's convolution inequality (also proved independently by Beckner in \cite{beckner}), and proved a generalized Young's inequality for more than three functions. Barthe gave a concrete description of when $\text{BL}(\vec{\pi},\vec{p})<\infty$ in the rank one case (all $n_j=1$) \cite{barthe}. In 2008, Bennett, Carbery, Christ, and Tao (BCCT) determined the criterion for $\text{BL}(\vec{\pi},\vec{p})<\infty$ \cite{BCCT}. BCCT also investigated extremal configurations and variants of the Brascamp-Lieb inequality in \cite{BCCT} and \cite{BCCT2} respectively.  

In this paper, we consider the following regularized version of (\ref{HBL}):
\begin{equation}
    \label{wHBL}\int_{[-R,R]^n} \prod_{j=1}^Jf_j(\pi_j(\x))^{p_j} d\x \le C \prod_{j=1}^J\left(\int_{\R^{n_j}}f_j(\x_j)d\x_j\right)^{p_j}
\end{equation}
where the $f_j$ are nonnegative, measurable functions that are constant on cubes in the unit cube lattice of $\R^{n_j}$, meaning on all sets of the form $\v+[0,1)^{n_j}$ where $\v\in\Z^{n_j}$. The optimal constant $\text{\underline{BL}}(\vec{\pi},\vec{p},R)$ for this inequality is defined by
\begin{equation}
    \text{\underline{BL}}(\vec{\pi},\vec{p},R):= \sup_{\vec{f}}\frac{\int_{[-R,R]^n}\prod\limits_{j=1}^Jf_j(\pi_j(\x))^{p_j}d\x}{\prod\limits_{j=1}^J\left(\int_{\R^{n_j}}f_j(\x_j)d\x_j\right)^{p_j}}\label{Cdef}
\end{equation}
where the supremum is taken over $\vec{f}=(f_1,\ldots,f_J)$ with nonzero, nonnegative functions $f_j\in L^1(\R^{n_j})$ that are constant on each cube in a unit cube tiling of $\R^{n_j}$. The goal of this paper is to identify the growth rate of $\text{\underline{BL}}(\vec{\pi},\vec{p},R)$  as a function of $R$, which we do in the following theorem. 
\begin{theorem}\label{main} Let $n\ge 1$, $J\ge 1$, and $\vec{p}\in[0,1]^J$. For each $j=1,\ldots,J$, let $n_j$ be an integer satisfying $1\le n_j\le n$ and let $\pi_j:\R^n\to\R^{n_j}$ be a surjective linear map. There exist constants $c,C\in\R^+$, depending on $\vec{\pi}$ and $\vec{p}$, which satisfy
\[c\sup_{V\le \R^n}R^{\dim V-\sum\limits_{j=1}^Jp_j\dim\pi_j(V)}\le \text{\emph{\underline{BL}}}(\vec{\pi},\vec{p},R)\le C\sup_{V\le \R^n}R^{\dim V-\sum\limits_{j=1}^Jp_j\dim\pi_j(V)}\]
for all $R\ge 1$, where the supremum is taken over all subspaces $V$ of $\R^n$.
\end{theorem}
Since the value of $\text{\underline{BL}}(\vec{\pi},\vec{p},R)$ does not change if we replace $\R^{n_j}$ by $\pi_j(\R^n)$, it is no loss of generality to assume that the $\pi_j$ are surjective. The quantity $\text{\underline{BL}}(\vec{\pi},\vec{p},R)$ is finite because 
\begin{align*}
    \int_{\{x\in\R^n:|x|\le R\}}\prod_{j=1}^Jf_j&(\pi_j(\x))^{p_j}d\x \le c_nR^n\prod_{j=1}^J\|f_j\|_{\infty}^{p_j} \\
    &\le c_nR^n\prod_{j=1}^J\left(\sum_{\v\in\Z^{n_j}}\|f_j\|_{L^\infty(\v+[0,1)^{n_j})}\right)^{p_j}\\
    &=c_nR^n\prod_{j=1}^J\left(\int_{\R^{n_j}}f_j(\x_j)d\x_j\right)^{p_j}.
\end{align*}
Also note that if the supremum defining $\text{\underline{BL}}(\vec{\pi},\vec{p},R)$ were taken over nonnegative $f_j\in L^1(\R^{n_j})$ which were \emph{not} necessarily constant on the unit cube lattice of $\R^{n_j}$, then $\text{\underline{BL}}(\vec{\pi},\vec{p},R)$ is not necessarily finite (see Theorem 2.2 in \cite{BCCT2}).

One motivation for understanding the regularized Brascamp-Lieb inequality is to extend multilinear Kakeya results. First formulated in 2006, the multilinear Kakeya inequality is an $L^p$-measurement of how much tubes in different directions can overlap. Let $\T_j$ be a collection of $\delta-$neighborhoods $T_j$ of lines in $\R^n$ that are sufficiently parallel (independent of $\delta$) to the $e_j$-axis, where $e_j$ is the standard basis vector. The multilinear Kakeya inequality states that
\begin{equation}\label{background MK}
\int_{\R^n}\prod_{j=1}^n\big(\sum_{T_j\in\T_j}\chi_{T_j}\big)^{\frac{p}{n}}\le C \delta^n\prod_{j=1}^n\big(\#\T_j\big)^{\frac{p}{n}}
\end{equation}
for $\frac{n}{n-1}\le p< \infty$. For the endpoint exponent $p=\frac{n}{n-1}$ and the special case that the $T_j$ are parallel to the $e_j$, (\ref{background MK}) is just the Loomis-Whitney inequality, a special Brascamp-Lieb inequality. Thus multilinear Kakeya can be regarded as a generalization of Loomis-Whitney. 

Bennett, Carbery, and Tao proved (\ref{background MK}) in 2006 without an epsilon loss away from the endpoint (which implies the endpoint with an epsilon loss by H\"{o}lder's inequality) \cite{BCT}. In 2009, Guth eliminated the $\epsilon$ loss factor using an alternative proof involving Dvir's polynomial method approach to the analogous problem over finite fields \cite{guth2,dvir}. Guth reproved a weaker, truncated version of the multilinear Kakeya inequality with a simple induction on scales argument in \cite{guth2}. 

Since the multilinear Kakeya inequality is a generalization of Loomis-Whitney, it is natural to wonder if there are analogues of the multilinear Kakeya inequality which are generalizations of Brascamp-Lieb inequalities. Bennett, Bez, Flock, and Lee \cite{stabBL} proved a stability property of the Brascamp-Lieb constant $\text{BL}(\vec{\pi},\vec{p})$ which combined with Guth's induction on scales result, led to a truncated multilinear Kakeya inequality with $\epsilon-$losses for finite Brascamp-Lieb data (meaning $\text{BL}(\vec{\pi},\vec{p})<\infty$). Zhang further developed Guth's polynomial method proof in \cite{guth2} to prove that these $\epsilon-$losses could be removed \cite{zhang}. 

In this paper, we prove a multilinear Kakeya result without any assumptions on the Brascamp-Lieb data. Instead, a factor which controls the regularized Brascamp-Lieb constant appears in the upper bound. To describe our generalized (truncated) multilinear Kakeya inequality, fix some notation. Let $\vec{\pi_0}$ be a $J$-tuple of orthogonal projections from $\R^n$ to subspaces $(V_j^0)^\perp\subset\R^n$ where $V_j^0$ has dimension $n_j'$. Let $n_j=n-n_j'$. We will consider affine subspaces $V_j\subset\R^n$ which, modulo translations, are within a distance $\nu$ of the $V_j^0$. Here distance is the standard metric on the Grassmann manifold of $n_j'$-dimensional subspaces of $\R^n$.

\begin{theorem} \label{genKak} Fix $\vec{p}\in[0,1]^J$ and $J$ orthogonal projections $\pi_j^0$ from $\R^n$ onto $(V_j^0)^\perp$. Then there exists $\nu>0$ satisfying for every $\epsilon>0$, there exists $C_\epsilon>0$ such that
\[ \int_{[-1,1]^n}\prod_{j=1}^J\big(\sum_{T_j\in\T_j}\chi_{T_j}\big)^{p_j}\le C_\epsilon\delta^{-\epsilon+n}\sup_{V\le \R^n}\delta^{-\dim V+\overset{J}{\underset{j=1}{\sum}} p_j\dim\pi_j^0(V)}\prod_{j=1}^J(\#\T_j)^{p_j} \]
holds for all finite collections $\T_j$ of $\delta$-neighborhoods of $n_j'$-dimensional affine subspaces of $\R^n$ which, modulo translations, are within a distance $\nu$ of the fixed subspace $V_j^0$. 
\end{theorem}

Using a small variation of the induction-on-scales method of Guth \cite{guth}, Theorem \ref{genKak} follows from Theorem 1  and the locally-bounded result Theorem \ref{techK} below, which follows from work in \cite{stabBL}.  Theorem \ref{techK} shows that the asymptotic growth rate in $R$ of $\text{\underline{BL}}(\vec{\pi},\vec{p},R)$ is stable under perturbation of $\vec{\pi}$. A similar stability for the implicit constant $C$ in the upper bound of Theorem \ref{main} is less clear, but turns out not to be necessary for proving Theorem \ref{genKak}. 

The following theorem is a locally uniform bound for $\text{\underline{BL}}(\vec{\pi},\vec{p},R)$, but which only achieves the expected sharp growth rate in $R$ for certain exponents $p_j$. Let $n_j=\dim (V_j^0)^\perp$ and for each $r\in\{1,\ldots,n\}$, $\mc{J}_r=\{j\in\{1,\ldots,J\}:n_j\ge r\}$. 
\begin{theorem} \label{locbd}
Fix $\vec{p}\in[0,1]^J$ and $J$ orthogonal projections $\pi_j^0$ from $\R^n$ onto $(V_j^0)^\perp$. Then there exist $\nu>0$ and $C_0>0$ such that 
\[ \text{\emph{\underline{BL}}}(\vec{\pi},\vec{p},R)\le C_0R^{\sum\limits_{r=1}^n\max(1-\sum\limits_{j\in\mc{J}_r}p_j,0)}  \]
holds for all orthogonal projections $\pi_j:\R^n\to V_j^\perp$ where $V_j$ is an $n_j'$-dimensional subspace of $\R^n$ which is within a distance $\nu$ of $V_j^0$.
\end{theorem}

We remark that the power of $R$ in Theorem \ref{locbd} above matches that in Theorem \ref{main} (with $(\vec{\pi}_0,\vec{p})$) if $p_j\le\frac{1}{J}$ for all $j$, but can be larger in general. Indeed, a helpful reviewer pointed out the example in $\R^2$ for which $\pi_1(x,y)=x$, $\pi_2(x,y)=y$, and $\pi_3(x,y)=(x,y)$ with exponents $\vec{p}=(1/4,1,1/2)$. In this case the power of $R$ in Theorem \ref{main} is $R^{1/4}$ while in Theorem \ref{locbd} above, we have $R^{1/2}$. Also observe that Theorem \ref{genKak} above implies an $R^\e$-loss version of Theorem \ref{locbd} with the optimal growth rate $R^{\sup_{V\le\R^n}[\dim V-\sum_{j=1}^Jp_j\dim\pi_j^0(V)]}$. 

An earlier draft of this manuscript had a stronger version of Theorem \ref{locbd} which was not fully justified. The proof of Theorem \ref{genKak}, which was used in the later work \cite{decquad}, was edited in the present manuscript to no longer rely on Theorem \ref{locbd} at all (see \textsection \ref{K}). This adjustment was facilitated by private correspondence with the authors of \cite{decquad}, and I am particularly appreciative of Ruixiang Zhang for helping with the revised proof. 

In \textsection{\ref{lower}}, we describe the example given by BCCT in \textsection{5} of \cite{BCCT2}, which proves the lower bound in Theorem \ref{main}. The upper bound in Theorem \ref{main} follows as a corollary of the more technical Proposition \ref{tech} discussed in \textsection{\ref{upper}}. The proof of Proposition \ref{tech} follows the inductive arguments of BCCT in the proofs of Theorems 2.1 and 2.5 of \cite{BCCT2}. In \textsection{\ref{K}}, we discuss the proof of Theorem {\ref{genKak}} and Theorem \ref{locbd}. The author wishes to thank Larry Guth for bringing this problem to her attention and to thank Michael Christ for valuable conversations. 

\section{Funding}

The author was supported by an National Science Foundation Graduate Research Fellowship under Grant No. [DGE 1106400].

\section{Lower bound for $\text{\underline{{\normalfont BL}}}(\vec{\pi},\vec{p},R)$ \label{lower}}
\vspace*{.03in}

In this section, we describe the example given by BCCT in \textsection 5 from \cite{BCCT2}. We use this example to demonstrate that there exists $c>0$ satisfying
\[ c\underset{V\le \R^n}{\sup}R^{\dim V-\sum\limits_{j=1}^Jp_j\dim\pi_j(V)} \le \text{\underline{BL}}(\vec{\pi},\vec{p},R) . \]

\begin{proof}[Proof of the lower bound for $\text{\underline{\normalfont BL}}(\vec{\pi},\vec{p},R)$ from Theorem \ref{main}]

Fix a vector $\vec{p}\in[0,1]^J$ and surjective linear maps $\pi_j:\R^n\to\R^{n_j}$. Let $V\le \R^n$ be a subspace. For a subspace $W\subset\R^k$ let $P_W:\R^k\to W$ denote linear projection onto $W$. Define the collections 
\[ S_{j}=\{\v\in\Z^{n_j}:|P_{\pi_j(V)}(\v)|\le R+\sqrt{n}\quad\text{and}\quad |P_{\pi_j(V)^\perp}(\v)|\le 1+\sqrt{n}\} \]
and let $f_j$ be the indicator function of the set $\underset{\v \in S_j}{\cup} (\v+Q_j)$, where $Q_j=[0,1)^{n_j}$. 

Let $c_0>0$ be a constant we will define independently of $R$. Define the set $S\subset\R^n$ by 
\[ S=\{\x\in V:|P_V(\x)|\le c_0R\quad\text{and}\quad |P_{V^\perp}(\x)|\le c_0\}. \]
Let $|\pi_j|$ be the operator norm of $\pi_j$ so that $|\pi_j(\x)|\le |\pi_j||\x|$ for all $\x\in\R^n$ and $j=1,\ldots,J$. Choose $c_0=\min(\frac{1}{2|\pi_1|},\ldots,\frac{1}{2|\pi_J|},\frac{1}{\sqrt{2}})$, which guarantees that $S\subset[-R,R]^d$. Now verify that if $\x\in S$, then $f_j(\pi_j(\x))=1$ for all $j=1,\ldots,J$: Write $\x\in\R^n$ uniquely as $\x={\bf{v}}+{\bf{w}}$ where ${\bf{v}}\in V$ and ${\bf{w}}\in V^\perp$. By the definition of $S$, $|{\bf{v}}|\le c_0R$ and $|{\bf{w}}|\le c_0$, which implies that $|\pi_j({\bf{v}})|\le|\pi_j||{\bf{v}}|\le \frac{1}{2}R$ and $|\pi_j({\bf{w}})|\le |\pi_j||{\bf{w}}|\le \frac{1}{2}$. Then \[|P_{\pi_j(V)}(\pi_j(\x))|\le |\pi_j(\v))|+|P_{\pi_j(V)}(\pi_j({\bf{w}}))|\le R \]
and 
\[|P_{\pi_j(V)^\perp}(\pi_j(\x))|= |P_{\pi_j(V)^\perp}(\pi_j({\bf{w}}))|\le 1. \]
It follows from these displayed inequalities that  $\pi_j(\x)\in\underset{\v\in S_j}{\cup}(\v+Q_j)$, which 
holds for all $j=1,\ldots,J$. Putting everything together, we have
\begin{align*}
    Cc_0^nR^{\dim V}&\le  \int_{S}\prod_{j=1}^J(1)^{p_j}d\x\le  \int_{[-R,R]^n}\prod_{j=1}^Jf_j(\pi_j(\x))^{p_j}d\x\\
    &\le \text{\underline{BL}}(\vec{\pi},\vec{p},R)\prod_{j=1}^J\left(\int_{\R^{n_j}}f_j(\x_j)d\x_j\right)^{p_j} 
\end{align*}   
for a constant $C$ which depends on the dimension $n$. Using the inclusion 
\[ \underset{\v\in S_j}{\cup}(\v+Q_j) \subseteq\{\x_j\in\R^{n_j}:|P_{\pi_j(V)}(\x_j)|\le R+2\sqrt{n}\quad\text{and}\quad|P_{\pi_j(V)^\perp}(\x_j)|\le 1+2\sqrt{n}\},\]
and that $f_j$ is the indicator function of $\underset{\v\in S_j}{\cup}(\v+Q_j)$, obtain the final inequality
\[ Cc_0^{n}R^{\dim V}\le \text{\underline{BL}}(\vec{\pi},\vec{p},R)\left(\int_{\R^{n_j}}f_j(\x_j)d\x_j\right)^{p_j}\le  \tilde{C}^{p_1+\cdots+p_J}\text{\underline{BL}}(\vec{\pi},\vec{p},R) \prod_{j=1}^J R^{p_j\dim \pi_j(V)} \]
where $\tilde{C}$ is a constant depending on $n$. 
\end{proof}

\section{Upper bound for $\text{\underline{BL}}(\vec{\pi},\vec{p},R)$ \label{upper} }

The upper bound from Theorem \ref{main} follows as a corollary to Proposition \ref{tech} below. The proof of Proposition \ref{tech} proceeds in an analogous way as the proofs of Theorems 2.1 and 2.5 in \cite{BCCT2}. We introduce some notation before we state the proposition. Let $H\subset\R^k$, $H_j\subset\R^{k_j}$ be subspaces and let $\pi_j:H\to H_j$ be surjective linear maps. For each $j=1,\ldots,J$, fix a subset $\mc{L}^0_j\subset \Z^{k_j}$ which satisfies $H_j\subset\underset{\v\in\mc{L}_j^0}{\cup}(\v+Q_j)$, where $Q_j=[0,1)^{k_j}$.  For a parameter $A>0$ and a function $f\in L^\infty(H_j)$, define the quantity
\begin{equation}  \|f\|_{A,\mc{L}_j^0}:=\sum_{\v\in\mc{L}_j^0}\|f\|_{L^\infty((\v+AQ_j)\cap H_j)},\end{equation}
where $AQ_j=[0,A)^{k_j}\subset\R^{k_j}$.


\begin{prop}\label{tech} Let $\vec{p}\in[0,1]^J$. Suppose $H$, $\{H_j\}$, $\{\mc{L}_j^0\}$, and $\vec{\pi}$ are given as above. Then there exist parameters $A_j\ge 1$, for $j=1,\ldots,J$, and a constant $C>0$ such that
\[ \int_{\{\x\in H:|\x|<R\}}\prod_{j=1}^Jf_j(\pi_j(\x))^{p_j}d\x\le C\sup_{V\le H}R^{\dim V-\sum\limits_{i=1}^Jp_i\dim\pi_i(V)}\prod_{j=1}^J\|f_j\|_{A_j,\mc{L}_j^0}^{p_j} \]
for all  $R\ge 1$ and all nonnegative, measurable functions $f_j:H_j\to[0,\infty)$. 
\end{prop}

Granting this proposition, we first prove the upper bound from Theorem \ref{main}. 
\begin{proof}[Proof of the upper bound for $\text{\underline{\emph{BL}}}(\vec{\pi},\vec{p},R)$] 
Take $H=\R^n$, $H_j=\R^{n_j}$, and $\mc{L}_j^0=\Z^{n_j}$ in Proposition \ref{tech}. It follows from the proposition that there exist constants $A_j,C$ such that
\[ \int_{[-R,R]^n}\prod_{j=1}^Jf_j(\pi_j({\bf{x}}))d{\bf{x}}\le C\sup_{V\le \R^n}R^{\dim V-\sum\limits_{i=1}^Jp_i\dim\pi_i(V)}\prod_{j=1}^J\|f_j\|_{A_j,\Z^{n_j}}^{p_j}\]
for all $R\ge 1$ and all measurable functions $f_j:\R^{n_j}\to [0,\infty)$. Consider functions $f_j$ which are constant on cubes $\v+[0,1)^{n_j}$ where $\v\in\Z^{n_j}$. Then 
\begin{align*}
    \|f_j\|_{A_j,\Z^{n_j}}=\sum_{m\in\Z^{n_j}}\underset{x\in m+A_jQ_j}{\text{ess sup}}f_j(x) &\le \sum_{m\in\Z^{n_j}}\sum_{\substack{k\in\Z^{n_j}\\ k\in m+A_jQ_j}}f_j(k)\\
    &\le A_j^{n_j}\sum_{k\in\Z^{n_j}}f_j(k)=A_j^{n_j}\int_{\R^{n_j}}f_j(\x_j)d\x_j , 
\end{align*}
which finishes the proof.
\end{proof}

\begin{proof}[Proof of Proposition \ref{tech}] Without loss of generality, assume that $0<\|f_j\|_{1,\mc{L}_j^0}<\infty$ for all $j=1,\ldots,J$. If $\underset{V\le H}{\sup}[\dim V-\sum\limits_{j=1}^Jp_j\dim\pi_j(V)]\le 0$, then Proposition \ref{tech} follows from Theorem 2.5 in \cite{BCCT2}, so it suffices to assume that the supremum is positive. Fix the notation $n=\dim H$ and $n_j=\dim H_j$. We proceed by induction on the dimension $n$. 

Begin with the base case $n=1$. In this case, $n_j=1$ for all $j$. If $1\le \sum\limits_{j=1}^Jp_j$, then choose $\delta_j\in[0,p_j)$ so that $1=\sum\limits_{j=1}^J(p_j-\delta_j)$. 
Using the generalized H\"{o}lder's inequality,
\begin{align} 
\int_{\{x\in H:|x|<R\}}\prod_{j=1}^Jf_j(\pi_j(x))^{p_j}dx&\le\prod_{j=1}^J\left(\int_{H}f_j(\pi_j(x))^{p_j/(p_j-\delta_j)}dx\right)^{p_j-\delta_j} \nonumber \\
&= \prod_{j=1}^J|\pi_j|^{\delta_j-p_j}\left(\int_{H}f_j(x)^{p_j/(p_j-\delta_j)}dx\right)^{p_j-\delta_j} \nonumber 
\end{align} 
where $|\pi_j|$ is from the subsitution $u=\pi_j(x)$. For each $1\le j\le J$, 
\[ \int_Hf_j(x)^{p_j/p_j-\delta_j}dx\le \sum_{m\in\mc{L}_j^0}\|f_j\|_{L^\infty(m+[0,1))}^{p_j/(p_j-\delta_j)}\le \big(\sum_{m\in\mc{L}_j^0}\|f_j\|_{L^\infty(m+[0,1))}\big)^{p_j/(p_j-\delta_j)},\]
which finishes this case. 

The other case is that  $1>\sum\limits_{j=1}^Jp_j$. Let $s=1-\sum\limits_{j=1}^Jp_j$. Then by the generalized H\"{o}lder's inequality, 
\begin{align*} 
\int_{\{x\in H:|x|<R\}}\prod_{j=1}^Jf_j(\pi_j(x))^{p_j}dx&\le \left(\int_{\{x\in H:|x|<R\}}dx\right)^s\prod_{j=1}^J\left(\int_{H}f_j(\pi_j(x))dx \right)^{p_j} \\
&\le  (2R)^{1-\sum\limits_{i=1}^Jp_i}\prod_{r=1}^J|\pi_r|^{-p_r}\prod_{j=1}^J\|f_j\|_{1,\mc{L}_j^0}^{p_j}. 
\end{align*} 
This concludes the $n=1$ base case.

Next consider the case where the ambient dimension is $n>1$. Furthermore, assume we are in the subcase where there exists a proper subspace $W\le H$ such that
\[ \dim W-\sum_{i=1}^Jp_i\dim\pi_i(W)=\sup_{V\le \R^n}\left(\dim V-\sum_{i=1}^Jp_i\dim\pi_i(V)\right). \]
Define coordinates $\x=(x,y)\in W\oplus W^\perp=H$ and write
\begin{align}\label{2} \int_{\{x\in H:|\x|<R\}}\prod_{j=1}^Jf_j(\pi_j(\x))^{p_j}d\x\le \int_{W^\perp}\int_{\{x\in W:|x|<R\}}\prod_{j=1}^Jf_j(\pi_j(x)+\pi_j(y))^{p_j}dxdy . \end{align}
For each $j$, let $P_{\pi_j(W)}:H_j\to\pi_j(W)$ and $P_{\pi_j(W)^\perp}:H_j\to\pi_j(W)^\perp$ be the linear projections. Define the subsets $\mc{L}_j\subset\mc{L}_j^0$ to be the collections of $\v\in\mc{L}_j^0$ satisfying $|P_{\pi_j(W)^\perp}(\v)|\le \sqrt{k_j}$. Then $\mc{L}_j$ contains the collection of $\v\in\mc{L}_j^0$ such that $\v+Q_j$ has nonempty intersection with $\pi_j(W)$, from which it follows that $\underset{\v\in\mc{L}_j}{\cup}(\v+Q_j)$ covers $\pi_j(W)$. Since $\dim W<\dim H$, by the inductive hypothesis (using the linear maps $\left.\pi_j\right|_{W}$), there exist $C,A_j>0$ (depending on the maximum of $k,k_1,\ldots,k_J$, $\|\vec{\pi}\|$, and ${\vec{p}}$) such that
\begin{align}\label{3}  
\int_{\{W:|x|<R\}}\prod_{j=1}^Jf_j&(\pi_j(x)+\pi_j(y))^{p_j}dx \le CR^{\dim W-\sum\limits_{i=1}^Jp_i\dim \pi_i(W)}\prod_{j=1}^J\|f_j(\cdot+\pi_j(y))\|_{A_j,\mc{L}_j}^{p_j}    
\end{align} 
for all $y\in W^\perp$. 
Let $L_j=P_{\pi_j(W)^\perp}\circ\left.\pi_j\right|_{W^\perp}:W^\perp\to\pi_j(W)^\perp$. The maps $L_j$ are clearly surjective because $\pi_j(W\oplus W^\perp)=\pi_j(W)+\pi_j(W^\perp)=H_j=\pi_j(W)\oplus \pi_j(W)^\perp$. For each $y\in W^\perp$, define $u_j(y)\in \pi_j(W)$ by $u_j(y)=P_{\pi_j(W)}(\pi_j(y))$, so $\pi_j(y)=u_j(y)+L_j(y)$. Now analyze the quantities $\|f_j(\cdot+\pi_j(y))\|_{A_j,\mc{L}_j}$ on the right hand side of the displayed inequality above. 

\begin{align*}
    \|f_j(\cdot+\pi_j(y))\|_{A_j,\mc{L}_j}&=\|f_j(\cdot+u_j(y)+L_j(y))\|_{A_j,\mc{L}_j}\\
    &= \sum_{\v\in\mc{L}_j}\underset{\x_j\in (\v+A_jQ_j)\cap\pi_j(W)}{\text{ess sup}}f_j(\x_j+u_j(y)+L_j(y)) \\
    &= \sum_{\v\in\mc{L}_j+u_j(y)}\underset{\x_j\in (\v+A_jQ_j)\cap\pi_j(W)}{\text{ess sup}}f_j(\x_j+L_j(y)) 
\end{align*}
where in the last line, we used the fact that since $u_j(y)\in\pi_j(W)$, $\pi_j(W)+u_j(y)=\pi_j(W)$. Note that for each $\v\in\mc{L}_j$, the number of ${\bf{w}}\in\mc{L}_j$ which satisfy $({\bf{w}}+u_j(y)+A_jQ_j)\cap(\v+A_jQ_j)\not= \emptyset$ is controlled by a dimensional constant times $A_j^{k_j}$. This means that
\[ \sum_{\v\in\mc{L}_j+u_j(y)}\underset{\x_j\in (\v+A_jQ_j)\cap\pi_j(W)}{\text{ess sup}}f_j(\x_j+L_j(y)) \le CA_j^{k_j}\|f_j(\cdot+L_j(y))\|_{A_j,\mc{L}_j}. \]
Putting this together with (\ref{2}) and (\ref{3}), we have
\[ \int_{\{\x\in H:|\x|<R\}}\prod_{j=1}^Jf_j(\pi_j(\x))^{p_j}d\x\le \tilde{C}R^{\dim W-\sum\limits_{i=1}^Jp_i\pi_i(W)}\int_{W^\perp}\prod_{j=1}^J\|f_j(\cdot+L_j(y))\|_{A_j,\mc{L}_j}^{p_j} dy ,\]
where $\tilde{C}$ depends on dimensions, $A_j$, $\vec{\pi}$, and $\vec{p}$. The growth rate of $\tilde{C}$ is not a concern since the number of steps in our induction is finite. 
For $z\in\pi_j(W)^\perp$, define
\begin{align*} 
F_j(z):=\|f_j(\cdot+z)\|_{A_j,\mc{L}_j}.
\end{align*} 
It remains to bound
\[ \int_{W^\perp}\prod_{j=1}^JF_j(L_j(y))^{p_j}dy.  \]
Fix the subset $\mc{L}_j'\subset\mc{L}_j^0$ of $\v\in\mc{L}_j^0$ that satisfy $|P_{\pi_j(W)}(\v)|\le \sqrt{k_j}$. The subset $\mc{L}_j'$ has the property that $\underset{\v\in\mc{L}_j'}{\cup}(\v+Q_j)$ covers $\pi_j(W)^\perp$. Using that $L_j:W^\perp\to \pi_j(W)^\perp$ is surjective (discussed above), by the inductive hypothesis, there exist constants $C',A_j'$ such that
\begin{align} \label{noR} \int_{\{y\in W^\perp:|y|<R\}}\prod_{j=1}^JF_j(L_j(y))^{p_j}dy\le C'\sup_{V\le W^\perp} R^{\dim V-\sum\limits_{i=1}^Jp_i\dim L_i(V)}\prod_{j=1}^J\|F_j\|_{A_j',\mc{L}_j'}^{p_j}. \end{align}
For any subspace $V\subset W^\perp$, 
\[ \dim(V+W)=\dim(V)+\dim(W),\]
and for each $i$,
\[ \dim\pi_i(V+W)=\dim L_i(V)+\dim\pi_i(W). \]
Using the above equalities and the inequality
\[ \dim (V+W)-\sum_{i=1}^Jp_i\dim\pi_i(V+W)\le \dim W-\sum_{i=1}^Jp_i\dim\pi_i(W),\]
which holds because of the subcase we are in, conclude that $\dim V\le \sum\limits_{i=1}^Jp_i\dim L_i(V)$ for all $V\le W^\perp$. This means that the power of $R$ that appears on the right hand side of (\ref{noR}) is 0, and we may let $R$ tend to infinity on the right hand side to integrate over all of $W^\perp$. Summarizing our results for this subcase, we have
\[ \int_{\{\x\in H: |\x|<R\}}\prod_{j=1}^Jf_j(\pi_j(\x))^{p_j}d\x\le\tilde{C}R^{\dim W-\sum\limits_{i=1}^Jp_i\dim\pi_i(W)}\prod_{j=1}^J\|F_j\|_{A_j',\mc{L}_j'}^{p_j}. \]
Observe that for each $j$,
\begin{align}
\sum_{k\in\mc{L}_j'} \underset{y\in (k+A_j'Q_j)\cap\pi_j(W)^\perp}{\text{ess sup}}F_j(y)&=\sum_{k\in\mc{L}_j'}\underset{y\in (k+A_j'Q_j)\cap\pi_j(W)^\perp}{\text{ess sup}}\sum_{m\in\mc{L}_j}\underset{x\in(m+A_jQ_j)\cap\pi_j(W)}{\text{ess sup}}f_j(x+y)\nonumber\\
    &\le \sum_{\substack{k\in\mc{L}_j'\\m\in\mc{L}_j}}\underset{x+y\in(k+m+(A_j+A_j')Q_j)\cap H_j}{\text{ess sup}}f_j(x+y). \label{4}
\end{align}
For each $\v\in\mc{L}_j^0$, the number of $(k,m)\in\mc{L}_j'\times\mc{L}_j$ such that \[(k+m+(A_j+A_j')Q_j)\cap(\v+(A_j+A_j')Q_j)\cap H_j\not=\emptyset\]
is bounded by 
\[ \#\{(k,m):\max(|P_{\pi_j(W)^\perp}(\v)-k|,|P_{\pi_j(W)}(\v)-m|)<2\sqrt{k_j}(1+A_j+A_j')\},\]
which is controlled by a constant depending only on dimension, $A_j$, and $A_j'$. Also use the fact that $\underset{\v\in\mc{L}_j^0}{\cup}(\v+Q_j)$ covers $H_j$ to bound the quantity in (\ref{4}) by 
\[ C\sum_{\v\in\mc{L}_j^0}  \underset{\x_j\in(\v+(A_j+A_j')Q_j)\cap H_j}{\text{ess sup}}f_j(\x_j)=C\|f_j\|_{A_j+A_j',\mc{L}_j^0},  \]
which finishes this subcase.

Now suppose that the ambient dimension $n>1$ and that all proper subspaces $W\le H$ satisfy 
\[ n-\sum_{i=1}p_in_i>\dim W-\sum_{i=1}^Jp_i\dim\pi_i(W). \]
Consider the set $K$ of all $J$-tuples $\vec{p}=(p_1,\ldots,p_J)\in [0,1]^J$ such that 
\[ n-\sum_{i=1}^Jp_in_i\ge \dim V- \sum_{i=1}^Jp_i\dim\pi_i(V)\]
for all proper subspaces $V\subset\R^n$. Equivalently, $K$ equals the intersection 
\[ K=[0,1]^J\cap \left(\underset{V\le\R^n}{\bigcap}\{\vec{p}\in\R^J:n-\sum_{i=1}^Jp_in_i\ge \dim V- \sum_{i=1}^Jp_i\dim\pi_i(V)\} \right).\]
This is the intersection of $[0,1]^J$ with finitely many closed half-spaces (even though there are infinitely many subspaces $V$, there are finitely many vectors $(\dim\pi_i(V))_i$). Therefore $K$ has finitely many extreme points, and since $K$ is compact and convex, $K$ equals the convex hull of its extreme points. If we prove the result for the extreme points $\vec{p}_k$, an application of complex interpolation says that if $\vec{p}=\sum\limits_k\lambda_k\vec{p}_{k}$ where $\sum\limits_k\lambda_k=1$, then
\begin{align*}  
\int_{H\cap|\x|<R}\prod_{j=1}^Jf_j(\pi_j(\x))^{p_j}d\x&\le C\prod_{k}\sup_{V\le H}R^{\lambda_k(\dim V-\sum\limits_{i=1}^Jp_{k,i}\dim\pi_{i}(V))}\prod_{j=1}^J\|f_j\|_{A_j,\mc{L}_j^0}^{\lambda_kp_{k,j}} \\
&=  C\sup_{V\le H}R^{\dim V-\sum\limits_{i=1}^Jp_{i}\dim\pi_{i}(V)}\prod_{j=1}^J\|f_j\|_{A_j,\mc{L}_j^0}^{p_j},
\end{align*} 
so it suffices to consider the extreme points. 

Next we want to show that $K=\tilde{K}$ where 
\[\tilde{K}=[0,\infty)^J\cap \left(\underset{V\le\R^n}{\bigcap}\{\vec{p}\in\R^J:n-\sum_{i=1}^Jp_in_i\ge \dim V- \sum_{i=1}^Jp_i\dim\pi_i(V)\} \right). \]
Observe $K=[0,1]^J\cap\tilde{K}\subset\tilde{K}$. Consider $\vec{p}\in\tilde{K}$ and any index $r$. Then for the subspace $V=\ker\pi_r$, $\vec{p}$ satisfies
\[ n-\sum_{i=1}^Jp_in_i\ge \dim\ker\pi_r-\sum_{i=1}^Jp_i\dim\pi_i(\ker\pi_r).\]
This leads to
\[ n_r=n-\dim\ker\pi_r\ge \sum_{i:i\not=r}p_i(n_i-\dim\pi_i(\ker\pi_r))+p_rn_r\ge p_rn_r ,\]
which means $p_r\le 1$ and thus $\tilde{K}\subset K$.

If $\vec{p}$ is an extreme point of $K$, then at least one of the inequalities defining $K$ must be equality. If 
\[ n-\sum_{i=1}^Jp_in_i=\dim V-\sum_{i=1}^Jp_i\dim \pi_i(V)\]
for some proper subspace $V\le H $, then we are in subcase 1 above. The alternative is that $p_r=0$ for at least one $r\in\{1,\ldots,J\}$. In that case, we are trying to prove that there exist $C,A_j$ so that
\[ \int_{\{\x\in H:|\x|<R\}}\prod_{j:j\not=r}f_j(\pi_j(\x))^{p_j}d\x\le C\sup_{V\le H}R^{\dim V-\sum\limits_{i:i\not=r}p_i\dim \pi_i(V)}\prod_{j:j\not=r}\|f_j\|_{A_j,\mc{L}_j^0}^{p_j} .\]
This follows from induction on the number of factors $J$ with base case $J=1$. When $J=1$, for appropriate constants $c,C>0$ depending on $\pi_J$, 
\[ \int_{\{\x\in H :|\x|\le R\}}f_J(\pi_J(\x))^{p_J}d\x =  CR^{n-n_J}\int_{\{\x_J\in H_J:|\x_J|\le c R\}}f_J(\x_J)^{p_J}d\x_J. \]
By H\"{o}lder's inequality, this is bounded by
\[ \tilde{C}R^{n-n_J}R^{n_J(1-p_J)}\left(\int_{H_J}f_J(\x_J)d\x_J\right)^{p_j}\le\tilde{C}R^{n-p_Jn_J}\|f_J\|_{1,\mc{L}_J^0}^{p_J} \]
for an appropriate constant $\tilde{C}$ and where we used that $\underset{m\in\mc{L}_j^0}{\cup}(m+Q_j)$ covers $H_j$ in the final inequality. 

\end{proof}

\section{ A generalized multilinear Kakeya inequality \label{K}}

In order to prove Theorem \ref{genKak}, we first prove a stability condition in the form of Theorem \ref{techK}. Theorem \ref{genKak} is proved using the induction-on-scales approach of Guth \cite{guth}.

Begin with Theorem \ref{techK}, which records the local boundedness of the growth rate of $R$ in $\text{\underline{BL}}(\vec{\pi},\vec{p},R)$ as a function of $\vec{\pi}$. Theorem \ref{techK} states a version of the result in \textsection 2.1 in \cite{stabBL} and we reproduce their argument (and some notations) here for completeness. 

\begin{theorem}[\cite{stabBL}]\label{techK} Fix $\vec{p}\in[0,1]^J$ and $J$ orthogonal projections $\pi_j^0$ from $\R^n$ onto $(V_j^0)^\perp$. Then there exists $\nu>0$ such that 
\[  \sup_{V\le \R^n}\big[\dim V-\sum_{j=1}^Jp_j\dim\pi_j(V)\big]\le\sup_{W\le \R^n}\big[\dim W-\sum_{j=1}^Jp_j\dim\pi_j^0(W)\big] \]
holds for all orthogonal projections $\pi_j:\R^n\to V_j$ where $V_j$ is an $n_j'$-dimensional subspace of $\R^n$ which is within a distance $\nu$ of the fixed subspace $V_j^0$.

\end{theorem}

\begin{proof} Let $1\le k\le n$ and let $\mc{E}_k$ denote the compact set of all orthonormal sets ${\bf{e}}=\{e_1,\ldots,e_k\}$ of vectors in $\R^n$. Fix $k$ and let $\bf{e}\in\mc{E}_k$. For each $1\le j\le J$, choose a subset $I_j\subseteq \{1,\ldots,k\}$ satisfying $|I_j|=\dim\langle \pi_j^0(e_1),\ldots,\pi_j^0(e_k)\rangle)$, 
\[ k-\sum_{j=1}^Jp_j\dim\pi_j^0(\langle e_1,\ldots,e_k\rangle)=k-\sum_{j=1}^Jp_j|I_j|, \]
and 
\[ \bigwedge_{i\in I_j}\pi_j^0(e_i)\not= 0. \]
Since $(\vec{\pi},{\bf{e}'})\mapsto\bigwedge_{i\in I_j}\pi_j(e_i')\in\Lambda^{|I_j|}(\R^{n_j})$ is continuous for each $j$, there exists $\e(\bf{e}),\delta(\bf{e})>0$ such that $\bigwedge_{i\in I_j}\pi_j(e_i')\not=0$ for each $j$ whenever $\|\bf{e}'-\bf{e}\|<\e(\bf{e})$ and $\|\vec{\pi}-\vec{\pi}_0\|<\delta(\e)$. In particular, $\dim\langle \pi_j(e_1'),\ldots,\pi_j(e_k')\rangle)\ge |I_j|$ for each $j$, so 
\[ k-\sum_{j=1}^Jp_j\dim\pi_j(\langle e_1',\ldots,e_k'\rangle)\le k-\sum_{j=1}^Jp_j\dim\pi_j^0(\langle e_1,\ldots,e_k\rangle). \]
Recall that $\mc{E}_k$ is compact, so there exists a finite collection ${\bf{e}}^1,\ldots,{\bf{e}}^N\in\mc{E}_k$ so that the sets 
\[ \{{\bf{e}}'\in\mc{E}_k:\|{\bf{e}}'-{\bf{e}}^{\ell}\|<\e({\bf{e}}^{\ell})\}\]
with $\ell=1,\ldots,N$ cover $\mc{E}_k$. Finally, choosing $\nu=\min\{\delta({\bf{e}}^1,\ldots,{\bf{e}}^N\}$, conclude that if $\|\vec{\pi}-\vec{\pi}_0\|<\nu$ (in the sense of Theorem \ref{techK}), then for any ${\bf{e}}\in\mc{E}_k$, there is some ${\bf{e}}^{\ell}$ from compactness so that
\[ k-\sum_{j=1}^Jp_j\dim\pi_j(\langle e_1,\ldots,e_k\rangle)\le k-\sum_{j=1}^Jp_j\dim\pi_j^0(\langle e_1^\ell,\ldots,e_k^\ell\rangle). \]
Since the number of $k$ is bounded by dimension, Theorem \ref{techK} is proved.

\end{proof} 

Next, we prove a weaker version of Theorem \ref{genKak} in the form of the following proposition. This proposition follows the argument of Guth in \cite{guth}. It is similar to the argument described in \cite{stabBL}, except that we do not use any locally uniform control over the (regularized) Brascamp-Lieb constants. An analogous proof could have been used in \cite{stabBL} to prove their Theorem 1.2 without relying on their Theorem 1.1. 

\begin{prop} \label{Kakstep1} Fix $\vec{p}\in[0,1]^J$ and $J$ orthogonal projections $\pi_j^0$ from $\R^n$ onto $(V_j^0)^\perp$. There exists $\nu>0$ so that the following is true. Let $\pi_j:\R^n\to V_j^\perp$ be orthogonal projection maps, where the $V_j$ are $n_j'$-dimensional subspaces within $\nu$ of the $V_j^0$. For every $\e>0$, there exist $\nu(\e,\vec{\pi})>0$ and $C(\e,\vec{\pi})<\infty$ so that
\[ \int_{[-1,1]^n}\prod_{j=1}^J\big(\sum_{T_j\in\T_j}\chi_{T_j}\big)^{p_j}\le C(\e,\vec{\pi})\delta^{-\epsilon+n}\sup_{V\le \R^n}\delta^{-\dim V+\overset{J}{\underset{j=1}{\sum}} p_j\dim\pi_j^0(V)}\prod_{j=1}^J(\#\T_j)^{p_j} \]
holds for all finite collections $\T_j$ of $\delta$-neighborhoods of $n_j'$-dimensional affine subspaces of $\R^n$ which, modulo translations, are within a distance $\nu(\e,\vec{\pi})$ of the fixed subspace $V_j$. 
\end{prop}

\begin{proof} Let $\epsilon>0$ be given and let $\nu>0$ be given by Theorem \ref{techK}. Suppose that $\pi_j:\R^n\to V_j^\perp$ are given as in the statement of Proposition \ref{Kakstep1}. 

We follow the multi-scale argument of Guth which analyzes the following quantity. Define $D(\delta,\w)$ to be the smallest constant in the inequality 
\begin{equation} \nonumber \int_{[-1,1]^n}\prod_{j=1}^J\big(\sum_{T_j\in\T_j}\chi_{T_j}\big)^{p_j}\le D_{}(\delta,\w)\prod_{j=1}^J(\delta^{n_j}\#\T_j)^{p_j} \end{equation}
where the $\T_j$ are arbitrary finite collections of $\delta$-neighborhoods of $\w$-perturbations of $V_j^{}$. The multi-scale relationship relates $D_{}(\delta,\w)$ to $D_{}(\delta/\w,\w)$.

Partition $[-1,1]^n$ into axis parallel cubes $Q$ of sidelength $\delta/\w$. Then
\begin{align}  
\int_{[-1,1]^n}\prod_{j=1}^J\big(\sum_{T_j\in\T_j}\chi_{T_j}\big)^{p_j}&=\sum_Q\int_Q\prod_{j=1}^J\big(\sum_{T_j\in\T_j}\chi_{T_j\cap Q}\big)^{p_j}\label{ok2} . 
\end{align} 
Let $T_j'$ be a $O(\delta)$-neighborhood of an affine subspace which is parallel to $V_j^{}$ and so that $T_j\cap Q\subset T'_j\cap Q$. For each $Q$,
\begin{align}
\int_{Q}\prod_{j=1}^J\big(\sum_{\substack{T_j\in\T_j\\T_j\cap Q\not=\emptyset}}\chi_{T_j'\cap Q}\big)^{p_j}&\le  \int_{Q}\prod_{j=1}^J\big(\sum_{\substack{T_j\in\T_j\\T_j\cap Q\not=\emptyset}}\chi_{\pi_j^{}(T_j'\cap Q)}\circ\pi_j^{} \big)^{p_j}\nonumber\\
&\le \delta^n \text{\underline{BL}}(\vec{\pi_{}},\vec{p},\w^{-1})\prod_{j=1}^J\big(\sum_{\substack{T_j\in\T_j\\T_j\cap Q\not=\emptyset}}\delta^{-n_j}|\pi_j^{}(T_j'\cap Q)|\big)^{p_j}. \label{ok}
\end{align}
where we used that $\delta^{-1}Q$ is a cube in $\R^n$ of sidelength $\w^{-1}$ and $\delta^{-1}T_j$ is a 1-neighborhood of an $n_j'$-dimensional affine subspace of $\R^n$, and so Theorem \ref{main} applies. Now let $\tilde{T}_j=T_j+B(0,c\delta/\w)$, where $c>0$ is large enough so that $T_j'\cap Q\not=\emptyset$ implies $Q\subset\tilde{T}_j$. Then for each $x_Q\in Q$, $\delta^{-n_j}|\pi_j^{}(T_j'\cap Q)|\lesssim \chi_{\tilde{T}_j}(x_Q)$, so
\[ \prod_{j=1}^J\big(\sum_{\substack{T_j\in\T_j\\T_j\cap Q\not=\emptyset}}\delta^{-n_j}|\pi_j^{}(T_j'\cap Q)|\big)^{p_j}\lesssim |Q|^{-1}\int_Q \prod_{j=1}^J\big(\sum_{\substack{T_j\in\T_j\\T_j\cap Q\not=\emptyset}}\chi_{\tilde{T}_j}\big)^{p_j} .    \]
Using this in (\ref{ok}) and with (\ref{ok2}) leads to 
\begin{align*}
    \int_{[-1,1]^n}\prod_{j=1}^J&\big(\sum_{T_j\in\T_j}\chi_{T_j}\big)^{p_j}\lesssim \delta^n\text{\underline{BL}}(\vec{\pi_{}},\vec{p},\w^{-1})\sum_Q(\delta/\w)^{-n}\int_Q\prod_{j=1}^J\big(\sum_{\substack{T_j\in\T_j}}\chi_{\tilde{T}_j}\big)^{p_j}    \\
    &= \w^{n}\text{\underline{BL}}(\vec{\pi_{}},\vec{p},\w^{-1})\int_{[-1,1]^n}\prod_{j=1}^J\big(\sum_{\substack{T_j\in\T_j}}\chi_{\tilde{T}_j}\big)^{p_j} \\
&\le \w^{n}\text{\underline{BL}}(\vec{\pi_{}},\vec{p},\w^{-1})D_{}(\delta/\w,\w)\prod_{j=1}^J\big((\delta/\w)^{n_j}\#\T_j)\big)^{p_j} 
\end{align*}
This proves that for a constant $\kappa>0$ which is allowed to depend on $\vec{\pi}$, 
\[ D(\delta,\w)\le \kappa \w^{n-\sum p_jn_j}\text{\underline{BL}}(\vec{\pi_{}},\vec{p},\w^{-1})D_{}(\delta/\w,\w). \]

Now use the multi-scale inequality. Iterate it $\ell$ times to obtain 
\begin{equation} \label{iteration} D_{}(\delta,\w)\le \kappa^\ell (\w^\ell)^{n-\sum p_jn_j}\text{\underline{BL}}(\vec{\pi_{}},\vec{p},\w^{-1})^\ell D_{}(\delta/\w^\ell,\w).\end{equation}
By Theorem \ref{main} and Theorem \ref{techK}, because $\vec{\pi_{}}$ is a $\nu$-perturbation of $\vec{\pi_0}$, there exists $C=C(\vec{\pi})$ so that 
\[ \text{\underline{BL}}(\vec{\pi_{}},\vec{p},\w^{-1})\le C_{}\sup_{V\le\R^n}\w^{-\dim V+\sum p_j\dim\pi_j^0(V)}.  \]
It follows from this and (\ref{iteration}) that
\begin{equation} \label{iteration2} 
D_{}(\delta,\w)\le (C_{}\kappa)^\ell (\w^\ell)^{n-\sum p_jn_j}\sup_{V\le\R^n}(\w^{-\ell})^{\dim V-\sum p_j\dim\pi_j^0(V)} D_{}(\delta/\w^\ell,\w).\end{equation}
Choose $\w$ so that $\w^{-\epsilon}=C_{}\kappa$ and take $\ell$ so that $\delta/\w^\ell\ge 1\ge \delta/\w^{\ell-1}$. It follows using the trivial bound $D(\delta/\w^\ell,\w)\le1 $ that
\begin{align*} 
D(\delta,\w)&\le \w^{-\e\ell} (\w^{\ell})^{n-\sum p_jn_j}\sup_{V\le\R^n}(\w^{-\ell})^{\dim V-\sum p_j\dim\pi_j(V)} \\
&\le \delta^{-\epsilon+n-\sum p_jn_j}\sup_{V\le\R^n}\delta^{-\dim V+\sum p_j\dim\pi_j^0(V)}.  
\end{align*} 
This proves the proposition, where $\nu(\epsilon,\vec{\pi})=(C\kappa)^{-\epsilon^{-1}}$.

\end{proof}

Finally, we prove Theorem \ref{genKak} using Proposition \ref{Kakstep1}. 

\begin{proof}[Proof of Theorem \ref{genKak}] Let $\epsilon>0$ be given and let $\nu=\frac{\nu'}{2}$, where $\nu'>0$ is given by Theorem \ref{techK} (which is the same $\nu'>0$ from Proposition \ref{Kakstep1}). This parameter $\nu$ depends on the fixed subspaces $V_j^0$. We finish the proof of Theorem \ref{genKak} using a compactness argument. 

Let $\pi_j^\theta:\R^n\to (V_j^\theta)^\perp$ denote orthogonal projection maps, where the $V_j^\theta$ are $n_j'$-dimensional subspaces within $\nu$ of the $V_j^0$. For each $\vec{\pi}_\theta$, let the parameters  $\nu(\e,\theta)=\nu(\e,\vec{\pi}_\theta)>0$ be given by Proposition \ref{Kakstep1} and write $R(\frac{\nu(\e,\theta)}{2},\vec{\pi}_\theta)$ to denote the open rectangle which is the product of $\frac{\nu(\e,\theta)}{2}$-neighborhoods of $V_j^\theta$ with respect to the Grassmanian metric on $n_j'$-dimensional subspaces. 

Let $K_j$ denote the closed $\nu$-neighborhood of $V_j^0$ and consider the $J$-fold product $K_1\times\cdots\times K_J$, which is compact. The open rectangles $R(\frac{\nu(\e,\theta)}{2},\vec{\pi}_\theta)$ cover $K_1\times\cdots\times K_J$, so we may extract a finite subcover indexed by $\theta_l$. Let $\nu_\e=\frac{1}{10}\underset{l}{\min} \,\nu(\e,\theta_l)>0$. For each $l$, cover $R(\frac{\nu(\e,\theta_l)}{2},\vec{\pi}_{\theta_l})$ by a finite set of $R(\nu_\e,\vec{\pi}_{\theta_l^m})\subset R(\frac{\nu(\e,\theta_l)}{2},\vec{\pi}_{\theta_l})$. 

The number of $\theta_l^m$ needed for the $R(\nu_\e,\vec{\pi}_{\theta_l^m})$ to cover $K_1\times\ldots\times K_J$ is determined by fixed dimensions, $\nu$, and the numbers $\nu(\epsilon,\theta)$. Since the set of $\nu(\epsilon,\theta)$ themselves depends on the fixed $V_j^0$, $\nu$, and $\epsilon$, this number of $\theta_l^m$ is an acceptable contribution to the $C_\epsilon$ constant in the statement of Theorem \ref{genKak}. 

We have covered $K_1\times\cdots \times K_J$ by at most $C_\e$ many $\nu_\e$-cubes, each of which is contained in some $R(\nu(\e,\theta_l),\vec{\pi}_{\theta_l})$ (where there are at most $C_\e$ many $l$). 
Write $\T_{j,\theta_l^m}$ for the collection of $T_j\in\T_j$ which are also in the $\nu_\e$-neighborhood of $V_j^{\theta_l^m}$. Then by pigeonholing,
\begin{align} 
\int_{[-1,1]^n}\prod_{j=1}^J\big(\sum_{T_j\in\T_j}\chi_{T_j}\big)^{p_j}&\lesssim_\epsilon \int_{[-1,1]^n}\prod_{j=1}^J\big(\sum_{T_j\in\T_{j,\theta_{l_j}^{m_j}}}\chi_{T_j}\big)^{p_j}.  \label{final}
\end{align} 
The tuple $(V_1^{\theta_{l_1}^{m_1}},\ldots,V_J^{\theta_{l_J}^{m_J}})$ is contained in $K_1\times\cdots\times K_J$, and therefore there is some $\theta_{l_0}^{m_0}$ so that $V_j^{\theta_{l_j}^{m_j}}$ is within $\nu_\e$ of $V_j^{\theta_{l_0}^{m_0}}$ for each $j$. Then $\T_{j,\theta_{l_j}^{m_j}}$ is a collection of $\delta$-neighborhoods of affines subspaces which are within $2\nu_\e$ of $V_j^{\theta_{l_0}^{m_0}}$ for each $j$. Since the rectangle of $\nu_\e$-neighborhoods of $V_j^{\theta_{l_0}^{m_0}}$ is contained in the rectangle of $\frac{\nu(\e,\theta_{l_0})}{2}$-neighborhoods of $V_j^{\theta_{l_0}}$, and by definition, $\nu_\e<\frac{\nu(\e,\theta_{l_0})}{2}$, conclude that $\T_{j,\theta_{l_j}^{m_j}}$ is a collection of affine subspaces within $\nu(\e,\theta_{l_0})$ of $V_j^{\theta_{l_0}}$ for each $j$. This means that we may control the right hand side of \eqref{final} by Proposition \ref{Kakstep1}. Since there are $C_\epsilon$ many $\theta_l^m$, we are done.

\end{proof}

\section{Some stability of $\underline{\text{{\normalfont BL}}}(\vec{\pi},\vec{p},R)$}

In this section, we describe a concrete stability result using the factoring-through-critical subspaces argument from \cite{BCCT}. In order to keep track of the implicit constant, we factor through one-dimensional subspaces at a time, sacrificing the probable optimal growth rate in $R$ except for special cases of exponents $\vec{p}$. 
Let $H$ be a vector space of dimension $n$, $V_j^0$ be $n_j'$-dimensional subspaces of $H$, and $n_j=n-n_j'$. 

\begin{notation} Let $V_j^0\subset\R^{n}$ be  nontrivial $n_j'-$dimensional subspaces and let $H_j^0=(V_j^0)^\perp$, which is $n_j$-dimensional.  
\end{notation}

\begin{notation} Let $V_j\subset\R^{n}$ also be $n_j'-$dimensional subspaces and write $H_j=V_j^\perp$, which is $n_j$-dimensional.
\end{notation}

\begin{definition}Suppose that $\pi_j^0:\R^n\to H_j^0$ and $\pi_j:\R^n\to H_j$ 
are orthogonal projection maps. For $\nu>0$, $\|\vec{\pi}-\vec{\pi}_0\|<\nu$ means that for each $j\in\{1,\ldots,J\}$, $V_j$ is within a distance $\nu$ of $V_j^0$ in the usual Grassmanian metric on $n_j'$-dimensional subspaces. 
\end{definition}

\begin{definition}
For each $j=1,\ldots,J$, fix a subset $\mc{L}^0_j\subset \Z^{n}$ which satisfies $H_j\subset\underset{\v\in\mc{L}_j^0}{\cup}(\v+Q)$, where $Q=[0,1)^{n}$.  For a subspace $W_j\subset H_j$ and a function $f\in L^\infty(H_j)$, define the quantity
\begin{equation} \label{norm} \|f\|_{W_j}:=\sum_{\v\in\mc{L}_j^0}\|f\|_{L^\infty((\v+Q)\cap W_j)}.\end{equation}

\end{definition}

\begin{notation} For vectors $w_1,\ldots,w_r\in\R^n$, let $\langle w_1,\ldots,w_r\rangle$ denote the linear span of $w_1,\ldots,w_r$. For each $j$, $\langle \pi_j^0(w_1),\ldots,\pi_j^0(w_r)\rangle^\perp$ will denote the orthogonal complement of $\langle w_1,\ldots,w_r\rangle$ in $H_j^0$ (and the same for $\pi_j,H_j$).

\end{notation}

\begin{lemma}\label{algo} Fix $\vec{\pi}_0$. There exists a basis $e_1,\ldots,e_n$ of $\R^n$, a small parameter $\nu>0$, and a constant $C>0$ such that the following holds. If $\vec{\pi}$ satisfies $\|\vec{\pi}_0-\vec{\pi}\|<\nu$, then for each $j\in\{1,\ldots,J\}$,
\[|\pi_j(e_1)|\ge C\]
and
\[ P_{\langle \pi_j(e_1),\ldots,\pi_j(e_i)\rangle^\perp}(\pi_j(e_{i+1}))=0\quad\text{or}\quad |P_{\langle \pi_j(e_1),\ldots,\pi_j(e_i)\rangle^\perp}(\pi_j(e_{i+1}))|\ge C \]
for all $i\in\{1,\ldots,n-1\}$. Furthermore, for each $j\in\{1,\ldots,J\}$ and $i\in\{2,\ldots,n\}$,
\begin{equation}\label{imp} \pi_j^0(e_i)\not\in\langle \pi_j^0(e_1),\ldots,\pi_j^0(e_{i-1})\rangle \implies \pi_j(e_i)\not\in\langle\pi_j(e_1),\ldots,\pi_j(e_{i-1})\rangle. \end{equation}

\end{lemma}

\begin{proof} We describe an algorithm for choosing the vectors $e_1,\ldots,e_n$. 

Choose a unit vector $e_1\in \R^n$ so that $e_1\not\in V_j^0$ for all $j=1,\ldots,J$. Then $\pi_j^0(e_1)\not=0$ for all $j$, so for $\nu>0$ sufficiently small,
\[ |\pi_j(e_1)|\ge \frac{1}{2}\min_i|\pi_i^0(e_1)|>0 \]
whenever $\|\vec{\pi}_0-\vec{\pi}\|<\nu$.

Now suppose that $e_1,\ldots,e_r$ have been chosen to satisfy the properties in the lemma until $i=r$. Let $J_r\subset\{1,\ldots,J\}$ be the set of indices $j$ for which $\text{dim}\langle\pi_j^0(e_1),\ldots,\pi_j^0(e_r)\rangle<n_j$. First suppose that $J_r$ is nonempty. Choose a unit vector $e_{r+1}\in \R^n$ such that $e_{r+1}\not\in\underset{j\in J_r}{\cup}\text{Span}(e_1,\ldots,e_r,V_j^0)$ (recall that $V_j^0=\ker\pi_j^0$). Then for each $j\in J_r$, $\pi_j^0(e_{r+1})\not\in\langle \pi_j^0(e_1),\ldots,\pi_j^0(e_r)\rangle$. This is because if there were constants $c_s$ so that $\pi_j^0(e_{r+1})=c_1\pi_j^0(e_1)+\cdots+c_r\pi_j^0(e_r)$, then $e_{r+1}-(c_1e_1+\cdots+c_re_r)\in V_j^0$, which contradicts how we selected $e_{r+1}$. Thus for each $j\in J_r$,
\[ P_{\langle \pi_j^0(e_1),\ldots,\pi_j^0(e_r)\rangle^\perp}(\pi_j^0(e_{r+1}))\not=0. \]
For sufficiently small $\nu>0$ and appropriate $C>0$, this also guarantees that
\[ P_{\langle \pi_j(e_1),\ldots,\pi_j(e_r)\rangle^\perp}(\pi_j(e_{r+1}))|\ge C\]
whenever $\|\vec{\pi}_0-\vec{\pi}\|<\nu$ and $j\in J_r$, and verifies property \eqref{imp} for $i=r+1$. If $j\not\in J_r$, then $\text{dim}(\pi_j^0(e_1),\ldots,\pi_j^0(e_r)\rangle=n_j$. For sufficiently small $\nu>0$, we also have $\text{dim}(\langle\pi_j(e_1),\ldots,\pi_j(e_r)\rangle=n_j$ whenever $\|\vec{\pi}_0-\vec{\pi}\|<\nu$ and $j\not\in J_r$. Then $\langle\pi_j(e_1),\ldots,\pi_j(e_r)\rangle=V_j^\perp$, so
\[ P_{\langle \pi_j(e_1),\ldots,\pi_j(e_r)\rangle^\perp}(\pi_j(e_{r+1}))=0 .\]

If $J_r$ is empty, then choose $e_{r+1},\ldots,e_n$ to be any orthonormal basis of $\langle e_1,\ldots,e_r\rangle^\perp$ in $\R^n$ and the induction ends. 
Since $J_r$ is empty, for each $j\in\{1,\ldots,J\}$, $\text{dim}\langle\pi^0_j(e_1),\ldots,\pi_j^0(e_r)\rangle=n_j$ and for sufficiently small $\nu>0$, $\text{dim}\langle\pi_j(e_1),\ldots,\pi_j(e_r)\rangle=n_j$ whenever $\|\vec{\pi}_0-\vec{\pi}\|<\nu$. Thus 
\[ P_{\langle \pi_j(e_1),\ldots,\pi_j(e_{s-1})\rangle^\perp}(\pi_j(e_s))=0\]
for each $j\in\{1,\ldots,J\}$ and $s\in\{r+1,\ldots,n\}$.

\end{proof}

\begin{definition}  For each $r\in\{1,\ldots,J\}$, let $\mc{J}_r=\{j\in\{1,\ldots,J\}:n_j\ge r\}$. 
\end{definition}

\begin{proof}[Proof of Theorem \ref{locbd}] In this argument, all constants are permitted to depend on $\vec{\pi}_0$ and $\vec{p}$, but not on $\vec{\pi}$. Consider the integral
\begin{equation}\label{startingpt} \int_{[-R,R]^n}\prod_{j=1}^Jf_j(\pi_j({\bf{x}}))^{p_j}d{\bf{x}}.   \end{equation}
Let $e_1,\ldots,e_n$ and $\nu>0$ be as in Lemma \ref{algo}. Express $\x$ in terms of this basis as $\x=y_1+\cdots+y_n$, where $y_i\in\langle e_i\rangle$. Perform a change of variables and use Fubini's theorem to bound the above integral by 
\begin{equation}\label{step0'}C\int_{\{|y_n|\le cR\}}\cdots\int_{\{|y_1|\le cR\}}\prod_{j=1}^Jf_j(\pi_j(y_1)+\cdots+\pi_j(y_n))^{p_j}dy_1\cdots dy_n     \end{equation}
where the constants $c,C>0$ only depend on $\vec{\pi}_0$ and the specific basis $e_1,\ldots,e_n$ we fixed from Lemma \ref{algo}. Each set $\{|y_i|\le cR\}$ is a subset of $\langle e_i\rangle$. For each $\underline{y}=y_2+\cdots+y_n\in\langle e_2,\ldots,e_n\rangle$, consider the innermost integral
\begin{equation*} \int_{\{|y_1|\le cR\}}\prod_{j=1}^Jf_j(\pi_j(y_1)+\pi_j(\underline{y}))^{p_j}dy_1. \end{equation*}
Let $H_j^1=\langle\pi_j(e_1)\rangle$. Apply the 1-dimensional base case in the proof of Proposition \ref{tech} to obtain the upper bound 
\[ 2(cR)^{\max(1-\sum\limits_{j=1}^Jp_j,0)}\prod_{r=1}^J|\pi_r(e_1)|^{-\g(p_r)}\prod_{j=1}^J\|f_j(\cdot+\pi_j(\underline{y})\|_{H_j^1}^{p_j}. \]
The exponent $-\g(p_r)$ depends on $\vec{p}$ and is negative. Since each of the $|\pi_r(e_1)|$ are bounded below by Lemma \ref{algo}, there is a constant $C$ so that the previous displayed math is bounded by
\begin{equation}\label{step1'} CR^{\max(1-\sum\limits_{j=1}^Jp_j,0)}\prod_{j=1}^J\|f_j(\cdot+\pi_j(\underline{y})\|_{H_j^1}^{p_j} . \end{equation}

For each $j$, let $L_j^{(0)}=\pi_j$ and let
\[ L_j^{(1)}=P_{(H_j^1)^\perp}\circ\left.\pi_j\right|_{\langle e_2,\ldots,e_n\rangle}:\langle e_2,\ldots,e_n\rangle \to (H_j^1)^\perp. \] 
Here the notation $(H_j^1)^\perp$ is the orthogonal complement of $\langle\pi_j(e_1)\rangle$ inside of $H_j$. The maps $L_j^{(1)}$ are clearly surjective because $\pi_j(\R^n)=H_j^1+\pi_j(\langle e_2,\ldots,e_n\rangle)=H_j$. For each $\underline{y}\in \langle e_2,\ldots,e_n\rangle$, define $u_j^{(1)}(\underline{y})\in H_j^1$ by $u_j^{(1)}(\underline{y})=P_{H_j^1}(\pi_j(\underline{y}))$, so $\pi_j(\underline{y})=u_j^{(1)}(\underline{y})+L_j^{(1)}(\underline{y})$. Now analyze the quantities $\|f_j(\cdot+\pi_j(\underline{y})\|_{H_j^1}$ that appear in \eqref{step1'}. 
\begin{align*}
    \|f_j(\cdot+\pi_j(\underline{y}))\|_{H_j^1}&=\|f_j(\cdot+u_j^{(1)}(\underline{y})+L_j^{(1)}(\underline{y}))\|_{H_j^1}\\
    &= \sum_{\v\in\mc{L}_j^0}\underset{x\in (\v+Q)\cap H_j^1}{\text{ess sup}}f_j(x+u_j^{(1)}(\underline{y})+L_j^{(1)}(\underline{y})) \\
    &= \sum_{\v\in\mc{L}_j^0+u_j^{(1)}(\underline{y})}\underset{x\in (\v+Q)\cap H_j^1}{\text{ess sup}}f_j(x+L_j^{(1)}(\underline{y})) 
\end{align*}
where in the last line, we used the fact that since $u_j^{(1)}(\underline{y})\in H_j^1$, $H_j^1+u_j^{(1)}(\underline{y})=H_j^1$. Note that for each $\v\in\mc{L}_j^0$, the number of ${\bf{w}}\in\mc{L}_j^0$ which satisfy $({\bf{w}}+u_j^{(1)}(\underline{y})+Q)\cap(\v+Q)\not= \emptyset$ is controlled by the ambient dimension $n$. This means that
\[ \sum_{\v\in\mc{L}_j+u_j^{(1)}(\underline{y})}\underset{x\in (\v+Q)\cap H_j^1}{\text{ess sup}}f_j(\x_j+L_j^{(1)}(y)) \le C\|f_j(\cdot+L_j^{(1)}(y))\|_{H_j^1}. \]
Putting everything together so far, we have bounded \eqref{step0'} by 
\[ CR^{\max(1-\sum\limits_{j=1}^Jp_j,0)}\int_{\{|y_n|\le cR\}}\cdots\int_{\{|y_2|\le cR\}}\prod_{j=1}^J\|f_j(\cdot+L_j^{(1)}(\underline{y}))\|_{H_j^1}^{p_j} d\underline{y} ,\]
where $C$ depends on $\vec{\pi}_0$ and $\vec{p}$.

Now suppose that the integral \eqref{step0'} is bounded by a constant factor times
\begin{align}\label{compcase} 
R^{\sum\limits_{r=1}^{i}\max(1-\sum\limits_{j=1}^Jp_j\text{dim}(\langle L_j^{(r-1)}(e_{r})\rangle),\,0)}\int_{\{|y_n|\le cR\}}\cdots\int_{\{|y_{i+1}|\le cR\}}\prod_{j=1}^J\|f_j(\cdot+L_j^{(i)}(y))\|_{H_j^i}^{p_j} dy ,
\end{align} 
where $H_j^r=\langle\pi_j(e_1),\ldots,\pi_j(e_r)\rangle$. The maps $L_j^{(r)}:\langle e_{r+1},\ldots,e_n\rangle \to (H_j^r)^\perp$ are defined by 
\[ L_j^{(r)}=P_{(H_j^r)^\perp}\circ \left.\pi_j\right|_{\langle e_{r+1},\ldots,e_n\rangle} . \] 
By construction of the $e_1,\ldots,e_n$, either $L_j^{(i)}(e_{i+1})=0$ or $|L_j^{(i)}(e_{i+1})|>c$ for $c$ depending only on $\vec{\pi}_0$. Furthermore, if $L_j^{(i)}(e_{i+1})=0$, then $L_j^{(i)}(e_r)=0$ for all $r=i+1,\ldots,n$. In that case, for each $\underline{y}\in\langle e_{i+1},\ldots,e_n\rangle$,
\[ \|f_j(\cdot+L_j^{(i)}(\underline{y}))\|_{H_j^i}=\|f_j\|_{H_j^i} ,\]
so we may ignore this constant factor from the integral (\ref{compcase}) above. Assume without loss of generality that $|L_j^{(i)}(e_{i+1})|>c$ for each $j$. Then apply the one-dimensional base case of the proof of Proposition \ref{tech} to obtain for each $\underline{y}\in\langle e_{i+2},\ldots,e_n\rangle$
\begin{align*} 
\int_{\{|y_{i+1}|<cR\}}\prod_{j=1}^J&\|f_j(\cdot+L_j^{(i)}(y_{i+1})+L_j^{(i)}(\underline{y}))\|_{H_j^i} ^{p_j}dy_{i+1} \\
&\le CR^{\max(1-\sum\limits_{j}p_j\text{dim}(\langle L_j^{(i)}(e_{i+1})\rangle),0)}\prod_{j=1}^J\|\|f_j(s+t+L_j^{(i)}(\underline{y})\|_{H_j^i,s}\|_{\langle L_j^{(i)}(e_{i+1})\rangle,t}^{p_j}. \end{align*} 
The norms on the right hand side are equal to 
\begin{align} \label{fra}
\sum_{m\in\mc{L}_j^0}\underset{t\in(m+Q)\cap\langle L_j^{(i)}(e_{i+1})\rangle}{\text{ess sup}}\sum_{r\in\mc{L}_j^0}\underset{s\in (r+AQ)\cap H_j^i}{\text{ess sup}}f_j(s+t+L_j^{(i)}(\underline{y}))
\end{align}
It suffices to sum over lattice points $m\in \mc{L}_j\subset\mc{L}_0$ and $r\in \mc{L}_j'\subset\mc{L}_j^0$ which are within a distance $\sqrt{n}$ of $\langle L_j^{(i)}(e_{i+1})\rangle$ and $H_j^i$ respectively. Similar to an analogous step in the proof of Proposition \ref{tech}, (\ref{fra}) is bounded by 
\begin{align}   \sum_{\substack{r\in\mc{L}_j'\\m\in\mc{L}_j}}\underset{s+t\in(r+m+2Q)\cap H_j^{i+1}}{\text{ess sup}}f_j(s+t+L_j^{(i)}(\underline{y})) \label{4'}
\end{align}
where $H_j^{i+1}=\langle \pi_j(e_1),\ldots,\pi_j(e_{i+1})\rangle$. Here we used that both $H_j^i$ and $\langle L_j^{(i)}(e_{i+1})\rangle$ are contained in $H_j^{i+1}$.  
For each $\v\in\Z^n$, the number of $(r,m)\in\mc{L}_j'\times\mc{L}_j$ such that \[(r+m+2Q)\cap(\v+Q)\cap H_j^{i+1}\not=\emptyset\]
is controlled by a constant depending only on dimension.
Also use the fact that $\underset{\v\in\Z^n}{\cup}(\v+Q)$ covers $H_j^{i+1}$ to bound the quantity in (\ref{4'}) by a constant multiple of 
\[ \sum_{\v\in\Z^n}  \underset{x\in(\v+Q)\cap H_j^{i+1}}{\text{ess sup}}f_j(x+L_j^{(i)}(\underline{y}))=\|f_j(\cdot+L_j^{i}(\underline{y}))\|_{H_j^{i+1}}.   \]
This concludes the intermediate inductive step. 

Finally, consider the result of this induction. We have shown that \eqref{step0'} is bounded by a constant multiple of 
\begin{align*}
R^{\sum\limits_{r=1}^{n}\max(1-\sum\limits_{j=1}^Jp_j\text{dim}(\langle L_j^{(r-1)}(e_r)\rangle),\,0)}\prod_{j=1}^J\|f_j\|_{H_j}^{p_j}  .
\end{align*}
The nonnegative input functions $f_j$ in the definition of $\underline{\text{BL}}(\vec{\pi},\vec{p},R)$ are assumed to be constant on cubes in the unit cube lattice of $\R^{n_j}$, where we identify $H_j$ (with the induced metric from $\R^n$) isometrically with $\R^{n_j}$. Then for a dimensional constant $C$,
\[ \|f_j\|_{H_j}\le C\int_{H_j}f_j. \]

It remains to check that the exponent of $R$ is the desired quantity. By the construction in Lemma \ref{algo}, $\dim (\langle L_j^{(r-1)}(e_r)\rangle)=1$ if $n_j\ge r$ and $\dim (\langle L_j^{(r-1)}(e_r)\rangle)=0$ if $n_j<r$. Define indexing sets $\mc{J}_r:=\{j\in\{1,\ldots,J\}:n_j\ge r\}$. Then  for each $r\in\{1,\ldots,n\}$,
\[ \sum_{j=1}^Jp_j\dim (\langle L_j^{(r-1)}(e_r)\rangle )=\sum_{j\in\mc{J}_r}p_j.\]

Note in addition that if $p_j\le \frac{1}{J}$ for each $j$, then for each $r$, $1\le r\le n$,
\[ 1-\sum_{j=1}^Jp_j\dim L_j^{(r-1)}(e_r)\ge 0. \]
In this case, 
\begin{align*} 
n-\sum_{j=1}^Jp_j \dim n_j&=1-\sum_{j=1}^Jp_j\dim\pi_j(e_1)+(n-1)-\sum_{j=1}^Jp_j \dim L^1(\langle e_2,\ldots,e_n\rangle)\\
&=\cdots=\sum_{r=1}^n(1-\sum_{j=1}^Jp_j\dim L_j^{(r-1)}(e_r)).
\end{align*}

\end{proof}

\end{document}